\documentclass[11pt]{amsart} 
\usepackage[all]{xy}
\usepackage{amsmath}
\usepackage{amssymb}
\usepackage{amsthm}
\usepackage[pdfauthor={Luis Sola and Matei Toma},
pdftitle={Maximal rationally connected fibrations and movable curves},
pdfkeywords={algebraic variety, foliation, rationally connected variety},
pdfview={Fit}]{hyperref}

\newtheorem{Th}{Theorem}[section]
\newtheorem{Lem}[Th]{Lemma}
\newtheorem{Cor}[Th]{Corollary}
\newtheorem{Prop}[Th]{Proposition}

\newtheorem{Def}[Th]{Definition}
\newtheorem{Not}[Th]{Notation}

\newtheorem{Rem}[Th]{Remark}
\newtheorem{RemDef}[Th]{Remark and definition}
\newtheorem{Ex}[Th]{Example}

\def\cM{\mathcal M }
\def\cI{\mathcal I }

\renewcommand{\P}{\mathbb{P}}

\newcommand{\R}{\mathbb{R}}

\def\Sing{\mbox{Sing}}

\def\rank{\mbox{rank}}

\def\Hom{\mbox{Hom}}

\title{Maximal rationally connected fibrations and movable curves}
\author{Luis E. Sol\'a Conde \and Matei Toma}
\date{\today}
\thanks{The authors wish to thank the
Max-Planck-Institut f\"ur Mathematik in Bonn and
the Institut Elie Cartan Nancy for hospitality and financial support during the preparation of this paper. The collaboration between the authors was partially supported by the project MTM2006-04785 of the Spanish Government.
AMS Classification (2000): 14J99, 14D99; secondary: 53C12.}

\begin{document}
\maketitle

\begin{abstract}
A well known result of Miyaoka asserts that a complex projective manifold is uniruled if its cotangent bundle restricted to a general complete intersection curve is not nef. Using the Harder-Narasimhan filtration of the tangent bundle, it can moreover be shown that the choice of such a curve gives rise to a rationally connected foliation of the manifold. In this note we show that, conversely, a movable curve can be found so that the maximal rationally connected fibration of the manifold may be recovered as a term of the associated Harder-Narasimhan filtration of the tangent bundle.
\end{abstract}

\noindent

\section{Introduction}
Since the 1980's it has become evident that rational curves play a central role in birational algebraic geometry, bringing uniruled varieties -that is, varieties covered by rational curves- into the area of interest of many researchers. 

Several characterizations of uniruledness may be found in the literature, all of them relating this property with positivity properties of the tangent bundle. In 1987 Miyaoka and Mori proved that a complex projective manifold $X$ is uniruled if and only if $K_X$ is negative on a curve $C$ passing by $x$ for general $x\in X$ (cf. \cite{MM}). This result has been recently improved by Boucksom, Demailly, P\u{a}un and Peternell, who have proved in \cite{BDPP} that in fact $X$ is uniruled if and only if $K_X$ is not pseudoeffective.

Also in 1987, Miyaoka provided a more geometric criterion of uniruledness: a complex projective manifold $X$ is uniruled if the restriction of $T_X$ to a general complete intersection curve has a vector subbundle of positive degree (\cite{Mi}, see also \cite{KST}). In fact, under the above hypotheses, the lower terms of the Harder-Narasimhan filtration of $T_X$ with respect to the chosen complete intersection curve are rationally connected foliations. 

On the other side, one may associate to a uniruled variety $X$ its maximally rationally connected fibration (MRCF for brevity), constructed by Campana and by Koll\'ar, Miyaoka and Mori independently, see \cite{C} and \cite{KMM}. 

In this note we address the question whether the MRCF of $X$ may be recovered as one of the foliations provided by Miyaoka's construction for a suitably chosen complete intersection curve. Unfortunately it is not even clear whether one may always find a complete intersection curve on which $K_X$ has negative degree (see Example \ref{ex:surf}), so we have decided to work with a broader class of curves. 

Given a complex projective manifold $X$ we will denote by $N_{\mbox{\tiny amp}}$ the interior of the closed cone $\overline{\mbox{ME}(X)}$ generated by movable curves,
see \cite[Def.~1.3]{BDPP}.
One has a notion of slope and stability of torsion free sheaves on $X$ with respect to a class in  $\overline{\mbox{ME}(X)}$,
cf. \cite{CP}. Our main result is the following

\begin{Th}\label{th:main}
Let $X$ be a uniruled complex projective manifold, and let $G\subset T_X$ denote the foliation associated with its MRCF. There exists a class $\alpha\in N_{\mbox{\rm\tiny amp}}$ represented by a reduced movable curve $C$ verifying that
$G|_C$ is  ample and $G$ is a member of the Harder-Narasimhan filtration of  $T_X$ with respect to $\alpha$.
\end{Th}
As a corollary we get a positive answer to our initial question when $\dim X=2$.

\vspace{0.3cm}

\noindent{\bf Acknowledgement:} We wish to thank F. Campana for useful discussions.

\section{Preliminaries}

Along this paper $X$ will denote a complex projective manifold of dimension at least two.

In this section we first review known facts about stability with respect to movable classes. 
We refer the interested reader to \cite{CP} for further details. Finally we discuss the relation of that concept with 
the maximal rationally connected fibration of $X$.

\begin{Def}\label{def:mov}{\rm A curve $C\subset X$ is called {\it movable} if there exists an irreducible algebraic family of curves containing $C$ as a reduced member and   
dominating $X$. As usual, we denote by $N_1(X)$ the vector space of numerical classes of real $1$-cycles on $X$ and by $N^1(X)$ its dual. The closure in $N_1(X)$ of the convex cone generated by classes of movable curves 
will be denoted by $\overline{\mbox{ME}(X)}$, and its interior by $N_{\mbox{\tiny amp}}$. The elements of $\overline{\mbox{ME}(X)}$ are called {\it movable classes}.}
\end{Def}

By a theorem of Boucksom, Demailly, P\u{a}un and 
Peternell (cf. \cite{BDPP}), $\overline{\mbox{ME}(X)}$ equals the dual of the {\it pseudoeffective} cone, i.e. 
the closure in $N^1(X)$ of the convex cone generated by classes of effective divisors. The pseudoeffective cone 
may also be described as the cone of classes of $\R$-divisors represented by positive closed $(1,1)$-currents on $X$.

\begin{Def}\label{def:weakfree}
{\rm A movable class $\alpha$ is called {\it weakly free} if given any codimension $2$ subset $B\subset X$ there exists a representative $C$ of $\alpha$ verifying $C\cap B=\emptyset$.}
\end{Def}

\begin{Rem}\label{rem:free}
 {\rm A smooth curve $C\subset X$ is usually called {\it free} (cf. \cite[II.3]{Kol}) if $H^1(C,T_X|_C)=0$ and $T_X|_C$ is globally generated. By \cite[II.3.7]{Kol}, the class of a free curve $C$ is weakly free.}
\end{Rem}

We may associate a notion of stability to every movable class $\alpha$.

\begin{Def}\label{def:slope}
{\rm Given a movable class $\alpha\in\overline{\mbox{ME}(X)}$ and a torsion free sheaf $E$ on $X$, we define the slope of $E$ 
with respect to $\alpha$ in the usual way:
$$
\mu_{\alpha}(E):=\dfrac{c_1(E)\cdot\alpha}{\rank(E)}.
$$
When there is
no possible confusion about the class we are using, we will denote the slope simply by $\mu$. 

We say that a torsion free sheaf $E$ on $X$ is {\it stable} (resp. {\it semistable}) with respect to $\alpha$ if the slope of every coherent subsheaf of $E$ is smaller than (resp. smaller than or equal to) $\mu_\alpha(E)$.}
\end{Def}

Note that if $\alpha$ is the class of a complete intersection curve of ample divisors $H_1,\cdots,H_{n-1}$, 
these concepts coincides with the usual notions of stability and semistability with respect to the polarizations 
$H_1,\cdots,H_{n-1}$. 

The following remark is well known for the usual concept of stability, and the same proof can be applied verbatim in our setting:

\begin{RemDef}\label{rem:HNF}
{\rm Given a movable class $\alpha$ and a torsion free sheaf $E$ on a variety $X$, the set of slopes of torsion free subsheaves of $E$ has a maximum, and the saturated subsheaf 
of maximal rank 
for which the maximal slope is achieved is called the {\it maximally destabilizing subsheaf} of $E$. By recursion 
one obtains the {\it Harder-Narasimhan filtration} of $E$ with respect to $\alpha$, which is the only filtration
$$
\mbox{HN}^\alpha(E):\;0=E_0\subsetneq E_1\subsetneq\ldots\subsetneq E_r=E
$$
verifying that the quotients $E_i/E_{i-1}$ are torsion free semistable sheaves, and the sequence of slopes 
$\big(\mu_i(E):=\mu(E_i/E_{i-1})\big)_i$ is strictly decreasing.}
\end{RemDef}

If $\alpha=[C]$ we may also consider the
Harder-Narasimhan filtration of $E|_C$. Although it is not true in general that this filtration coincides with the restriction of $\mbox{HN}^\alpha(E)$, it does when $C$ is a general complete intersection curve. This follows from the
Mehta-Ramanathan restriction theorem (cf. \cite{MR}, see also \cite{F}) and as a consequence we obtain in this case that
$E_i|_C$ is ample if $\mu_i(E)>0$.

In this paper we are interested in the Harder-Narasimhan filtration of $T_X$ with respect to movable classes $\alpha$. Let us introduce the notation that we will use below.
\begin{Not}\label{not}
 {\rm Let $X$ be a smooth complex projective variety and $\alpha$ be a movable class in $X$. We will denote by $$\mbox{HN}(T_X):0=F_0\subsetneq F_1\subsetneq\ldots\subsetneq F_r=T_X$$ the Harder-Narasimhan filtration of $T_X$ with respect to $\alpha$ and by $\mu_i$ the slopes of $F_i/F_{i-1}$, for all $i$.
We will also denote by $s$ the integer $s=\max\big(\{0\}\cup\{i|\;\mu_i>0\}\big)$.}
\end{Not}

The first property we are concerned with is the behaviour of $\mbox{HN}(T_X)$ with respect to birational transformations:

\begin{Lem}\label{lem:birHNF}
Let $\alpha=[C]$ be a movable class on $X$, and let $\pi:\widetilde{X}\to X$ be projective birational morphism. Assume that $C$ does not intersect the image $B$ of the exceptional locus $E$ of $\pi$. Then the Harder-Narasimhan filtrations of $T_X$ with respect to $\alpha$ and of $T_{\widetilde{X}}$ with respect to the pull-back class $\widetilde{\alpha}$ of $\alpha$ coincide over $X\setminus B$. Thus the Harder-Narasimhan filtration of $T_X$ with respect to weakly free classes is birationally invariant.
\end{Lem}

\begin{proof}
 Let $0=F_0\subset F_1\subset\ldots\subset F_r=T_X$ denote the Harder-Narasimhan filtration of $T_X$ with respect to 
$\alpha$. For every $i$ let us denote by $\widetilde{F}_i$ the saturation in $T_{\widetilde{X}}$ of the extension by zero of 
$F_i|_{U}$, where $U$ denotes the open set $X\setminus B\cong \widetilde{X}\setminus E$. 

By construction $\mu_{\widetilde{\alpha}}(\widetilde{F}_i/\widetilde{F}_{i-1})=\mu_{\alpha}(F_i/F_{i-1})$, hence it suffices to show that $\widetilde{F}_i/\widetilde{F}_{i-1}$ is semistable with respect to $\widetilde{\alpha}$ for all $i$. In fact, given any destabilizing subsheaf $D\subset \widetilde{F}_i/\widetilde{F}_{i-1}$, the saturation in $T_X$ of the extension by zero of $D|_U$ would destabilize $F_i/F_{i-1}$, a contradiction.
\end{proof}

Using the fact that the tensor product of semistable sheaves, modulo torsion, is again semistable (\cite[Appendix]{CP}), one 
can easily prove that the $F_i$'s are integrable provided $i\leq s$. If $s>0$ and $\alpha$ is the class of a general complete intersection curve $C$, then $F_s|_C$ is ample and \cite{KST} allows us to say that the 
leaves of $F_s$ are algebraic and rationally connected. Unfortunately, since the ampleness of $F_s|_C$ is not known for a general movable curve $C$, the previous argument cannot be used in general.
Nevertheless the following property is still fulfilled:

\begin{Lem}\label{lem:contamp}
With the notation in \ref{not}, let $G$ be a subsheaf of $T_X$ and assume further that $\alpha=[C]$, where $C$ is a curve contained in the regular locus of the distributions $G$ and $F_s$ (that is, the locus where $T_X/G$ and $T_X/F_s$ are locally free). If moreover $G|_C$ is ample, then $G$ is contained in $F_s$.
\end{Lem}

\begin{proof}
Assume the contrary. By the ampleness of $G|_C$ we get $\mu(G+F_s/F_s)=\mu(G/G\cap F_s)>0$, therefore $\mu_{s+1}>0$, 
contradicting the choice of $s$.
\end{proof}

The following result will be used in the case of the MRCF.

\begin{Lem}\label{lem:notcontamp}
Let $X$ and $Z$ be smooth projective varieties of positive dimension, $\pi:X\dashrightarrow Z$ a rational map, $G$ the foliation on $X$ associated to $\pi$ and $\alpha $ a weakly free class on $X$. If  $G\subsetneq F^{\alpha}_s$, then $Z$ is uniruled. 
\end{Lem}

\begin{proof}
Since by Lemma \ref{lem:birHNF} the Harder-Narasimhan filtration of $T_X$ with respect to weakly free classes is birationally invariant, we may assume that $\pi$ is a morphism. The same property will allow us to reduce ourselves  to the situation when $\pi$ is also equidimensional. Indeed, over some open subset $V$ of $Z$ the morphism $\pi$ coincides with the universal family over some locally closed subvariety $U$ of the corresponding Chow scheme of $X$. Let $Z_1$ be the closure of $U$ and $X_1$ be the universal family over $Z_1$. Let further $Z_2\to Z_1$ denote a desingularization of $Z_1$ and $X_2$ the normalization of $X_1\times_{Z_1}Z_2$. We get the following diagram
   $$
   \xymatrix{X_2\ar[d]^{\pi_2}\ar@/^1pc/[rrr]^{\mbox{\tiny 
bir}}\ar[r]&X_1\ar[d]\ar@{-->}[r]&U\ar[d]\ar@{^{(}->}[r]&X\ar[d]^{\pi}\\
   Z_2\ar[r]&Z_1\ar@{-->}[r]&V\ar@{^{(}->}[r]&Z}
   $$
where the composition of the rational maps in the upper row is a birational morphism. The pullback of $\alpha$ by this morphism is a weakly free class on $X_2$. 

We now replace  $\pi:X\to Z$ by $\pi_2:X_2\to Z_2$ which is equidimensional and change notation accordingly. The fact that $X$ could be singular will not bother us since  $\alpha$ is   weakly free. 
   
As before the leaves of the foliation $F_s$ contain the leaves of $G$. 

We claim that the foliation $F_s$ descends to a foliation $E$ on $Z$. 
Indeed, consider a small analytic neighborhood $W$ around a general point $x$ of a fixed general fiber of 
 $\pi$.  
The foliation defined by $F_s$ is regular around $x$ and is a product in $W$.
Each leaf $L$ of $F_s$ in $W$ projects onto a smooth closed submanifold $S$ of $\pi(U)$ with $\dim S=\rank
(F_s/G)$. Therefore the closure of the leaf of $F_s$ in $\pi^{-1}(\pi(W))$ containing $L$ is $\pi^{-1}(S)$. Since
  $\pi^{-1}(\pi(W))$ is a neighbourhood of $\pi^{-1}(\pi(x))$ we see that $S$ is independent of the choice of $x
  \in\pi^{-1}(\pi(x))$. This defines the desired foliation on $Z$. 
  
  Choose now a movable curve $C$ representing $\alpha $ and avoiding  $\Sing X$, 
  the singular sets of the foliations $F_s$ and $G$ as well as the pull-back of the singular set of the foliation $E$. 
  Then in a neighbourhood of $C$ it is easy to check that $F/G$ and $\pi^*E$ are locally free 
  and that there exists an injective morphism $F/G\to \pi^*E$. 
  This implies $\deg_{\alpha}(\pi^*E)\geq \deg_{\alpha}(F_s/G)>0.$
Thus $E$ has positive degree with respect to the movable curve
$C'=\pi(C)$ and by 
  \cite[Thm.~1.9]{CP} it follows that
 $Z$ is uniruled. 
\end{proof}

\section{Movable curves and rationally connected foliations}

In this section we prove Theorem \ref{th:main} and make some comments on the two dimensional case.

\subsection{The rationally connected case}\label{ssec:ratcon}

For rationally connected varieties the statement of the main theorem takes the following form.
\begin{Prop}\label{baby} Let $X$ be a rationally connected manifold. There exists a class $\alpha\in N_{\mbox{\rm\tiny amp}}$ represented by a movable curve $C$ verifying that $T_X|_C$ is ample.
\end{Prop}

In order to prove this result we need some preparations.
\begin{Def}
{\rm Let $T$ be a positive closed current of bidegree $(1,1)$ on $X$ and $f:C\to X$ a non-constant morphism from a smooth connected curve $C$ to $X$. Choose local $i\partial\overline{\partial}$-potentials $u_i$ for $T$ on open sets $U_i\subset X$ and suppose that $f(C)\cap U_i$ is not contained in the polar set of $u_i$ for some $i$. Then define locally $f^*T$ as
$i\partial\overline{\partial}(u_i\circ f)$.}
\end{Def}
\begin{Lem}
   Under the above conditions $f^*T$ is a well defined positive closed current of bidegree $(1,1)$ on $C$ and
   $[f^*T]=f^*[T]$ where $[ \ ]$ denotes taking cohomology classes.
\end{Lem}

\begin{proof}
There exists a smooth closed $(1,1)$-form $\eta$ in the cohomology class of $T$. Therefore $T-\eta=i\partial\overline{\partial}v$. We may assume $U_i$ to be biholomorphic to open balls, hence $\eta|_{U_i}=i\partial\overline{\partial}v_i$ for some smooth functions $v_i$. Thus $f^*T$ is given locally as
$i\partial\overline{\partial}((v_i+v|_{U_i})\circ f)$ showing that $[f^*T]=[f^*\eta]=f^*[T]$.
\end{proof}
\begin{Def}\label{def:good}
{\rm We shall say that a (connected) family of curves $(C_s)_{s\in S}$ on $X$ is {\em good} if there exists some closed subset
$A\subset X$ of codimension two such that for all $x\in
 X\setminus A$ there exist $C_{s_1},...,C_{s_n}$ curves in the family passing through $x$ and analytically open
connected local branches $C_{s_i}^0$ of $C_{s_i}$ at $x$
 with the following properties
\begin{enumerate}
\item  $C_{s_i}^0$ is smooth at $x$ for all $i$,
\item $\sum_{i=1}^n T_xC_{s_i}^0=T_xX$,
\item for all $i$ there are $n-1$-dimensional subfamilies in $S$ of curves whose connected local branches give foliations locally around $x$ which contain  $C_{s_i}^0$ as a leaf.
\end{enumerate}}
\end{Def}

\begin{Lem}
If $(C_s)_{s\in S}$ is a good family of curves on $X$ and $T$ is a positive closed current of bidegree $(1,1)$ such that the intersection product
$[T][C_s]$ vanishes, then $T=0$.
\end{Lem}
\begin{proof}
Take $f:C'_s\to C_s$ the morphism of normalization. Since the family is good $f^*T$ is well defined for general $s\in S$ and  we see that the hypothesis $[T][C_s]=0$ implies $f^*T=0$.
Now adopting the notations of the above definition, let $x\in
 X\setminus A$ and $\pi_i:U_i\to B_i$ be a smooth map from an open neighborhood $U_i$ of $x$
 onto a $n-1$-dimensional base $B_i$ defining the local
  foliation which contains $C_{s_i}^0$  as a leaf. Let further $\omega_i$ be the pull-back to $U_i$ of a volume form
  on $B_i$. Then our assumption on $T$ provides  $T\wedge \omega_i=0$ on $U_i$ and thus
  $T\wedge \sum_{i=1}^n\omega_i=0$ around $x$. But $\sum_{i=1}^n\omega_i$ is a strictly positive $(n-1,n-1)$-form and thus
  the $(n-1)$-st power of a $(1,1)$-positive form $\Omega$ on $\cap_{i=1}^n U_i $, cf. \cite[Appendix]{CP}.
  Since the trace measure of $T$ with respect to $\Omega$ vanishes, the support of $T$ will be contained in $A$, hence
  $T=0$.
\end{proof}

\begin{Cor}\label{cor:good}
If $(C_s)_{s\in S}$ is a good family on $X$ then $[C_s]\in N_{\mbox{\rm\tiny amp}}$.
\end{Cor}
\begin{proof}
By \cite{BDPP} the class $\alpha$ of a curve belongs to $N_{\mbox{\tiny amp}}$ if and only if for every non-zero 
positive closed current $T$ on $X$ one has $[T]\alpha>0$. For any non-zero positive closed current $T$ 
on $X$ there exists an element $C_s$ such that $f^*T$ is well defined, where $f:C'_s\to C_s$ denotes normalization.
Then the above Lemma yields $[T][C_s]>0$.
\end{proof}

\begin{proof}[Proof of Proposition~\ref{baby}]
From the previous discussion, it is enough to show that any rationally connected manifold $X$ admits good 
families of very free rational curves. 

Take $x_1$ any point of $X$. By \cite[Thm.~IV.3.9]{Kol} there is an 
immersion $f$ of $\P^1$ such that $f(0)=x_1$ and $f^*T_X$ is ample. Let us denote by $\cM_1$ the family of curves determined by the component of the Hom-scheme $\Hom(\P^1,X)$ containing $[f]$. It verifies condition (1) of Definition \ref{def:good} at $x_1$ by construction. Arguing as in \cite[Prop.~2.3]{Hw} one sees that the tangent directions to general deformations of $f$ sending $0$ to $x_1$ fill in almost all of $\P(\Omega_X|_{x_1})$, hence we get condition (2). Finally, note that the differential of the evaluation map is surjective at $(f,0)$, by the ampleness of $f^*T_X$, hence the last condition at $x_1$ is satisfied as well. 

Let us denote by $U_1\subset X$ an open set where the family $\cM_1$ verifies (1), (2) and (3).
For any $x_2\in X\setminus U_1$ we can consider $\cM_2$ and an open set $U_2$ constructed as above. Consider now the family $\cM$ of curves of the form $C_1\cup C_2$, for $[C_i]\in\cM_i$, $i=1,2$. It verfies the properties (1), (2) and (3) for every point in $U_1\cup U_2$. By noetherian recursion the assertion follows.  
\end{proof}

\subsection{The general case}\label{ssec:general}

\begin{proof}[Proof of Theorem~\ref{th:main}]
Suppose that $X$ is uniruled but not rationally connected and denote by $G$ the foliation induced by the MRCF of $X$. We first show that a weakly free curve $C$ exists in $X$ such that $G|_C$ is ample and $[C]\in N_{\mbox{\tiny amp}}$. 

We take a general smooth complete intersection curve $D$ on $X$ avoiding the singularities of $G$. 
Following \cite[II~7]{Kol}, we will consider combs with handle $D$ and $m$ teeth $R_i$ which are smooth free rational curves 
contained in distinct general fibers of the rationally connected fibration,
each $R_i$ meeting $D$ transversely in 
one point $P_i$. By Proposition \ref{baby} we may also assume that $G|_{R_i}$ is ample for all $1\le i\le m$. By \cite[II~7.9-7.10]{Kol}, for arbitrarily large $m$ one can find such a comb $C_0$ which deforms to give a covering family $(C_t)_{t\in T}$ of curves on $X$
such that the general member $C_t$ is smooth and free. Moreover, $[C_t]$ is weakly free and the family $(C_t)_{t\in T}$ is good since the family of complete intersection curves containing $D$ was already good. In particular 
$[C_t]\in N_{\mbox{\tiny amp}}$ by Corollary \ref{cor:good}.

We now prove that $G|_{C_t}$ is ample. Let us denote by
$$
\xymatrix{S\ar[r]^q\ar[d]_p&X\\T}
$$
the desingularization of the smoothing described above, where $T$ is a smooth curve. In particular the central fiber $S_0\cong C_0\cup\bigcup_jE_j$, where $q(E_j)$ is zero-dimensional, hence $q^*(G|_{E_j})$ is trivial for every $j$.

We will consider the relative Harder-Narasimhan filtration 
$$
0=G_0\subsetneq G_1\subsetneq \dots\subsetneq G_k=q^*G
$$
of $q^*G$, see \cite[2.3]{HL}. Over some open subset $U\subset T$, it induces the usual Harder-Narasimhan filtration of $q^*(G|_{C_t})$, $t\in U$. Let $G'$ denote the quotient $q^*G/G_{k-1}$ modulo torsion. It will be enough to show that $\deg_{C_t}(G')$ is positive.

In fact, the morphism $q^*G\to G'$ induces a surjective morphism $$\bigwedge^r q^*G\to\det G'\otimes\cI_Z,$$ where $r$ denotes the rank of $G'$ and $\cI_Z$ is the ideal sheaf of a zero-dimensional subscheme $Z\subset S$. Now, for general $t\in U$
$$
\begin{array}{l}\vspace{0.2cm}\deg_{C_t}(G')=\deg_{C_t}(\det G')=\\=\deg_{D}(\det G')+\sum_{i=1}^m\deg_{R_i}(\det G')+\sum_{j}\deg_{E_j}(\det G').
\end{array}
$$
For each component $B$ of the curve $S_0$ we may consider the image $P_B$ of the natural map $\big(\det G'\otimes\cI_Z\big)|_B\to (\det G')|_B$, which is a quotient of $\left(q^*\bigwedge^rG\right)|_B$ and thus
$$
\deg_B(\det G')\geq\deg(P_B)\geq\mu_{\mbox{\tiny min}}\left(q^*\bigwedge^rG\right)|_B.
$$
The last term is nonnegative for $B=E_j$ and it is at least $1/r$ when $B=R_i$, hence 
$$
\deg_{C_t}(G')\geq \mu_{\mbox{\tiny min}}\left(\bigwedge^rG\right)|_D+\frac{m}{r}
$$
is strictly positive for $m>>0$.

Let as usual $\alpha=[C_t]$ and consider the Harder-Narasimhan filtration of $T_X$ with respect to $\alpha$. By Lemma \ref{lem:contamp} it follows that $G\subset F_s$. If $G\neq F_s$, then by Lemma \ref{lem:notcontamp} the target of the MRCF of $X$ would be uniruled. But this is impossible by \cite{GHS}, thus $G=F_s$ and the Theorem is proved.
\end{proof}

\subsection{Comments on surfaces}\label{ssec:surf}

In the case of surfaces we get the following 

\begin{Cor}
Let $X$ be a uniruled complex projective smooth surface and let $G$ denote the foliation associated with the MRCF of $X$. There exists an ample divisor $C$ on $X$ such that $G|_C$ is ample and $G$ is a member of the Harder-Narasimhan filtration of $T_X$ with respect to $[C]$.
\end{Cor}

\begin{proof}
Since $N_{\mbox{\tiny amp}}$ coincides in this case with the cone of ample divisors of $X$, the statement follows from Theorem \ref{th:main}.
\end{proof}

Note that even for a rational surface $X$, it is not obvious a priori how to find an ample divisor $C$ verifying $-K_X\cdot C>0$. The following example was shown to us by F. Campana.

\begin{Ex}\label{ex:surf}
{\rm Let $X$ be the blow-up of $\P^2$ at sixteen general points $P_1,\dots,P_{16}$ and let $C$ be the strict transform of a general quintic passing through $P_1,\dots,P_{16}$. It is easy to check that $C$ is a very ample divisor on $X$ but $-K_X\cdot C=-1$.}
\end{Ex}

%%%%%%%%%%%%%%%%%%%%%%%%%%%%%%%%%%%%%%%%%%%%%%%%%%%%%%%%%%%%%%%%%%%%%%%%%%%%%%%%%%%%%%%%%%%%%%%%%%%%%%%%%%%%%
%%%%%%%%%%%%%%%%%%%%%%%%%%%%%%%%%%%%%%%%%%%%%%%%%%%%%%%%%%%%%%%%%%%%%%%%%%%%%%%%%%%%%%%%%%%%%%%%%%%%%%%%%%%%%

\hrule \medskip
\par\noindent

Addresses:

\par\noindent{\it Matei Toma}: Institut de Math\'ematiques Elie Cartan,
Nancy-Universit\'e,\\
B.P. 239,
54506 Vandoeuvre-l\`es-Nancy Cedex,
France \\ and
Institute of Mathematics of the Romanian Academy.\\
{\tt toma@iecn.u-nancy.fr}

\par\noindent{\it Luis E. Sol\'a Conde}: Departamento de Matem\'aticas, Universidad Rey Juan Carlos,\\
C. Tulip\'an s.n. 28933, M\'ostoles, Madrid. Spain.\\
{\tt luis.sola@urjc.es}

\end{document}